\documentclass{amsart}

\usepackage{amsmath,mathtools,amssymb,textcomp,mleftright,stmaryrd,geometry}

\mleftright
\usepackage[T1]{fontenc}
\newcommand{\bbQ}{{\mathbb{Q}}}
\newcommand{\bbK}{{\mathbb{K}}}
\newcommand{\bbR}{{\mathbb{R}}}
\newcommand{\bbL}{{\mathbb{L}}}
\newcommand{\bbZ}{{\mathbb{Z}}}
\newcommand{\bbF}{{\mathbb{F}}}
\newcommand{\bbC}{{\mathbb{C}}}
\newcommand{\bbN}{{\mathbb{N}}}
\newcommand{\calR}{{\mathcal{R}}}
\newcommand{\calO}{{\mathcal{O}}}
\newcommand{\calD}{{\mathcal{D}}}
\newcommand{\calH}{{\mathcal{H}}}
\newcommand{\calN}{{\mathcal{N}}}
\newcommand{\frake}{{\mathfrak{e}}}
\newcommand{\frakf}{{\mathfrak{f}}}
\newcommand{\frakG}{{\mathfrak{G}}}

\newcommand{\zetax}{\zeta_{2(p-1)}}
\usepackage{dsfont}
\newcommand{\bbone}{\mathds{1}}
\newcommand{\DBerk}{\calD_{\mathrm{Ber},K}^1}
\DeclarePairedDelimiter\norm{\lVert}{\rVert}%

\usepackage{hyperref}
\usepackage[nameinlink]{cleveref}
\Crefname{conjecture}{Conjecture}{Conjectures}
\Crefname{lemma}{Lemma}{Lemmas}
\Crefname{definition}{Definition}{Definitions}
\Crefname{remark}{Remark}{Remarks}
\Crefname{proposition}{Proposition}{Propositions}
\Crefname{corollary}{Corollary}{Corollarys}
\Crefname{equation}{}{}
\Crefname{item}{}{}
\usepackage{xcolor}
\usepackage{orcidlink}
\usepackage[inline]{enumitem}
\newtheorem{theorem}{Theorem}[section]
\newtheorem*{theorem*}{Theorem}
\newtheorem{lemma}[theorem]{Lemma}
\newtheorem{remark}[theorem]{Remark}
\newtheorem{proposition}[theorem]{Proposition}
\newtheorem*{proposition*}{Proposition}
\newtheorem{definition}[theorem]{Definition}
\newtheorem{conjecture}[theorem]{Conjecture}
\newtheorem{corollary}[theorem]{Corollary}
\newtheorem{example}[theorem]{Example}

\numberwithin{equation}{section}
\usepackage{microtype}
\usepackage[anythingbreaks]{breakurl}
\title{Uniformizer of the false Tate curve extension of $\mathbb{Q}_p$ (II)}

\author{Shanwen Wang\orcidlink{0000-0003-0228-1208}}
\address{School of Mathematics, Renmin University of China, No. 59 Zhongguancun Street, Haidian District, Beijing, 100872, China}
\email{s\_wang@math.ruc.edu.cn}
\thanks{Shanwen Wang is supported by the Fundamental Research Funds for the Central Universities, the Research Funds of Renmin University of China No.2020030251 and The National Natural Science Foundation of China (Grant No.11971035).}

\usepackage{orcidlink}
\author{Yijun Yuan\orcidlink{0000-0001-6571-6980}}
\address{Yau Mathematical Sciences Center, Tsinghua University, Beijing, 100084, China}
\email{941201yuan@gmail.com}


\begin{document}
\begin{abstract}
    In this article, we investigate the explicit formulae for the uniformizers of the false Tate curve extension of ${\bbQ}_p$. More precisely, we establish the formulae for the fields ${\bbK}_p^{m,1}={\bbQ}_p(\zeta_{p^m}, p^{1/p})$ with $m\geq 1$ and for general $n\geq 2$, we prove the existence of the recurrence polynomials ${\calR}_p^{m,n}$ for general field extensions ${\bbK}_p^{m, n}$ of ${\bbQ}_p$, which shows the possibility to construct the uniformizers systematically.
\end{abstract}
\maketitle

\section{Introduction}
Fix a prime $p\geq 3$.
In this article, we continue to explore the explicit uniformizer of the false Tate curve extension $\bbK_p^{m,n}=\bbQ_p\left(\zeta_{p^m}, p^{1/p^n}\right)$ of $\bbQ_p$, for $n,m\in\bbN$. One of the interests of the construction of the explicit uniformizer of the false Tate curve extension is to study the field of norms (cf. \cite{FonWin1979a,FonWin1979b}) of the $p$-adic Lie extension $\bbQ_p^{\mathrm{FT}}=\cup_{m,n\geq 1}\bbK_p^{m,n}$ of $\bbQ_p$, which is still a mysterious (cf. \cite{BergerMOF}) and should play a role in the non-commutative Iwasawa theory for the tower of the false Tate curve extensions.

The objective of this article is in two folds:
\begin{enumerate}
    \item Construct the explicit uniformizers of the tower $\bbK_p^{m,1}$ for $m\geq 1$, which is ``orthogonal'' to our previous work (cf. \cite{WangYuan2021}) on $\bbK_p^{2,n}$ for $n\geq 1$.
    \item Investigate the formulae of the uniformizers of $\bbK_p^{m,n}$ for general $m,n$.
\end{enumerate}
\subsection{Uniformizers of $\bbK_p^{m,1}$}
We construct the uniformizers of $\bbK_p^{m,1}$ for $m\geq 2$ using the induction process as follows:
For $m=2$, we take the uniformizer
\[{\pi_p^{2,1}}=p^{-1/p}\left(\zeta_{p^2}-\sum_{k=0}^{p-1}\frac{1}{[k!]}\zetax^k p^{\frac{k}{p(p-1)}}\right)\]
of $\bbK_p^{2,1}$ constructed in \cite[Theorem 3.23]{WangYuan2021}, where $[\cdot]: \bar{\bbF}_p\rightarrow W(\bar{\bbF}_p)$ is the Teichm\"{u}ller character. Set two polynomials in $\bbZ_{(p)}[T]$
$$\calR_p^{3,1}(T)=\sum_{k=0}^{p-1}\frac{(-1)^k}{k!}T^k+\sum_{k=1}^{p-1}\frac{(-1)^k(k-H_k)}{k!}T^{p+k}$$
and
$$\calR_p^{\geq 4,1}(T)=\sum_{k=0}^{p-1}\frac{(-1)^k}{k!}T^k+\sum_{k=1}^{p-1}\frac{(-1)^k(2k-H_k)}{k!}T^{p+k},$$
where $H_k$ is the $k$-th harmonic number $\sum_{i=1}^k\frac{1}{i}$.
For $m=3$, let
\begin{equation}\label{eq:5235}
    \pi_p^{3,1}= \left(\zeta_{p^3}-1\right)^{-2p+2}\cdot\left(\zeta_{p^3}-\calR_p^{3,1}\left(\pi_p^{2,1}\right)\right)\in\bbK_p^{3,1}.
\end{equation}
For $m\geq 4$, we recursively set
\begin{equation}\label{eq:61052}
    \pi_p^{m,1}=
    \left(\zeta_{p^m}-1\right)^{-2p+2}\cdot
    \left(\zeta_{p^m}-\calR_p^{\geq 4,1}\left(\pi_p^{m-1,1}\right)\right)\in\bbK_p^{m,1}.
\end{equation}
Then, our first result is the following theorem, whose proof relies on the truncated expansion of $\zeta_{p^m}$ in the $p$-adic Mal'cev-Neumann field studied in \cite{WangYuan2022}
:
\begin{theorem}\label{thm:bestuniformizer}
    For an integer $m\geq 3$, $\pi_p^{m,1}$ is a uniformizer of $\bbK_p^{m,1}$.
\end{theorem}
\subsection{Recurrence polynomials for general case}
The appearance of the recurrence polynomials $\calR_p^{3,1}$ and $\calR_p^{\geq 4,1}$ motivates us to investigate the existence of the recurrence polynomials for general $n\geq 2$. Since we have already constructed the uniformizers $\pi_p^{2,n}$ of the field $\bbK_p^{2,n}$ in \cite{WangYuan2021}, the recurrence polynomials (if they exist) will give the formulae of the uniformizers of general $\bbK_p^{m,n}$.
Our second result is exactly the existence of such recurrence polynomials:
\begin{theorem}\label{thm:40345}
    For any integer $n\geq 1$, there exists a series\footnote{The series $\left\{\calR_p^{m,n}(T)\right\}_{m\geq 3}$ is obviously non-unique.} of polynomials $\left\{\calR_p^{m,n}(T)\right\}_{m\geq 3}$ in $\bbZ_{(p)}[T]$ and two series of integers $\{\alpha_p^{m,n}\}_{m\geq 3},\{\beta_p^{m,n}\}_{m\geq 3}$ such that the element
    $$\pi_p^{m,n}\coloneqq \left(\zeta_{p^m}-1\right)^{\alpha_p^{m,n}}\cdot\left(\zeta_{p^m}-\calR_p^{m,n}\left(\pi_p^{m-1,n}\right)\right)^{\beta_p^{m,n}}$$
    is a uniformizer of $\bbK_p^{m,n}$ for all $m\geq 3$.
\end{theorem}

Moreover, since for $n=1$ the recurrence polynomials are eventually stable (i.e. we have $\calR_p^{m,1}(T)=\calR_p^{\geq 4,1}(T)$ for all $m\geq 4$),  we make a conjecture on the  stability of the recurrence polynomials:
\begin{conjecture}
    For every $n\geq 1$, the series of recurrence polynomials $\left\{\calR_p^{m,n}(T)\right\}_{m\geq 3}$ in \Cref{thm:40345} can be chosen to be eventually stable, i.e. for $m$ sufficiently large, we have
    $$\calR_p^{m,n}=\calR_p^{m+1,n}=\calR_p^{m+2,n}=\cdots.$$
\end{conjecture}
\begin{remark}
    For the proof of \Cref{thm:40345}, we use the theory of inductive valuations of MacLane. This is inspired by the algorithm used by \texttt{\bfseries MAGMA} (cf. \cite{MR1484478}) and a developing \texttt{\bfseries SageMath} (cf. \cite{sagemath}) package \texttt{\bfseries henselization} (cf. \cite{RuthCode}), which compute a defining polynomial\footnote{For any $p$-adic field $K$, by primitive root theorem, $K$ is generated by a single element $\alpha$ over $\bbQ_p$. Then the defining polynomial is defined to be the minimal polynomial of $\alpha$ over $\bbQ_p$.} $\mathrm{Def}_f(T)$ of the splitting field $\mathrm{Split}(f)$ of a polynomial $f\in\bbQ_p[T]$ over $\bbQ_p$, where $\mathrm{Def}_f$ is guaranteed to be Eisenstein when $\mathrm{Split}(f)/\bbQ_p$ is totally ramified.
\end{remark}

\subsection*{Acknowledgement} The authors would like to express their gratitude to Tim Dokchitser, Maurizio Monge, Sebastian Pauli, Julian Rüth and Tudor Micu for numerous helpful discussions on Ore-MacLane algorithms through email and GitHub Issues (cf. \cite{RuthDiscussion}).

Shanwen Wang is supported by the Fundamental Research Funds for the Central Universities, the Research Funds of Renmin University of China No.2020030251 and The National Natural Science Foundation of China (Grant No.11971035).
\section{Uniformizer of $\bbK_p^{m,1}$}
The idea to construct the uniformizer of the false Tate curve extension $\bbK_p^{2,n}$ of $\bbQ_p$ for $n\geq 1$ in \cite[Theorem 3.23]{WangYuan2021} is to use the truncated expansion of $\zeta_{p^2}$ to construct an algebraic number with a ``nice'' $p$-adic valuation. Basically, we would like to use the same idea to construct the uniformizer of $\bbK_p^{m,n}$ for $m\geq 3$. Let $\sigma_n=\sum_{k=n}^{+\infty}p^{-\frac{1}{p^k}}$. By the recurrence relation for $\pi_p^{m,1}$, to prove \Cref{thm:bestuniformizer}, we are reduced to compute the truncated expansions of $\left(\zeta_{p^{m+1}}-1\right)^{-2p+2}$ and $\left(A_{p,m}^{(\beta)}\right)^k$ with
\[A_{p,m}^{(\beta)}=(-1)^m\zetax p^{\frac{1}{p^{m-1}(p-1)}}\sigma_m\left(1+(-1)^m\beta\zetax p^{\frac{1}{p^{m-1}(p-1)}}\right)+O\left(p^{\frac{2}{p^{m-1}(p-1)}}\right),\]
$\beta\in {\bbZ}_p$ and $1\leq k\leq 2p-1$.

To simplify the statement, for every integer $s\geq 1$, we represent the indicator function of the set $\{1,2,\cdots,s\}$ by $\bbone_{\leq s}(x)$, i.e.
\[\bbone_{\leq s}(x)=\begin{cases}
        1, & \text{ if }x\in\{1,2,\cdots,s\}; \\
        0, & \text{ otherwise}.
    \end{cases}\]
Besides that, the indicator function of the set $\{s\}$ is denoted by $$\bbone_{s}(x)=\bbone_{\leq 1}(x-s+1).$$

\subsection{$\left(\frac{2}{p^{n-1}(p-1)}\right)$-truncated expansion of $\zeta_{p^n}$ in the $p$-adic Mal'cev-Neumann field}
Let $\calO_{\breve{\bbQ}_p}=W(\bar{\bbF}_p)$ be the ring of Witt vectors over $\bar{\bbF}_p$ and let $\bbL_p$ be the $p$-adic Mal'cev-Neumann field $\calO_{\breve{\bbQ}_p}((p^{\bbQ}))$ (cf. \cite[Section 4]{Poonen1993}). Every element $\alpha$ of $\bbL_p$ can be uniquely written as
\begin{equation}
    \label{expansioncan}\sum_{x\in \bbQ}[\alpha_x]p^x.
\end{equation}
For any $\alpha=\sum_{x\in \bbQ}[\alpha_x]p^x\in \bbL_p$, we set $ \mathrm{Supp}(\alpha)=\{x\in \bbQ:  \alpha_x\neq 0\}$, which is well-orderd by the definition of $\bbL_p$.
Thus, we can define the $p$-adic valuation $v_p$ by the formula: $$v_p(\alpha)=\begin{cases} \inf \mathrm{Supp}(\alpha), & \text{ if } \alpha\neq 0; \\ \infty, & \text{if } \alpha=0.\end{cases} $$
The field $\bbL_p$ is complete for the $p$-adic topology, and it is also algebraically closed. Moreover, it is the maximal complete immediate extension\footnote{A valued field extension $(E,w)$ of  $(F,v)$ is an immediate extension, if  $(E,w)$ and $(F,v)$ have the same residue field. A valued field  $(E,w)$  is maximally complete if it has no immediate extensions other than $(F,v)$ itself.} of $\overline{\bbQ}_p$. The field $\bbL_p$ is spherical complete\footnote{A valued field is said to be spherical complete, if the intersection of every decreasing sequence of closed balls is nonempty.}, and the field $\bbC_p$ of $p$-adic complex numbers is not spherical complete, which can be continuously embedded into $\bbL_p$.
\begin{definition}\leavevmode
    \begin{enumerate}
        \item For any $\alpha\in \bbL_p$, we call the unique expression \Cref{expansioncan} of $\alpha$, the \textbf{canonical expansion} of $\alpha$.
        \item Let $r\in \bbQ$ and $\alpha\in \bbL_p$, we rewrite the canonical expansion of $\alpha$ in the following way
              \[\alpha=\sum_{x\in \bbQ, x< r  }[\alpha_x]p^x+ O(p^r),\]
              and we call this expression the $r$-\textbf{truncated canonical expansion} of $\alpha$.
        \item If $\sum_{x\in \bbQ, x< r  }\beta_xp^x+ O(p^r)$ is another element in $\bbL_p$ with $\beta_x\in \calO_{\breve{\bbQ}_p}$ such that
              \[\sum_{x\in \bbQ, x< r  }\beta_xp^x\equiv \alpha \bmod{p^r},\] then we call $\sum_{x\in \bbQ, x< r  }\beta_xp^x+ O(p^r)$ a $r$-\textbf{truncated expansion} of $\alpha$.
    \end{enumerate}
\end{definition}

The following $\frac{2}{p^{n-2}(p-1)}$-truncated expansion of $\zeta_{p^n}$ is established in \cite[Theorem 1.6]{WangYuan2022}:
\begin{proposition}\label{truncatedfinal}
    For $n\geq 2$, we have the following $\frac{2}{p^{n-2}(p-1)}$-truncated expansion of $\zeta_{p^n}$:
    \begin{align*}
        \zeta_{p^n}= & \sum_{k=0}^{p-1}\frac{(-1)^{nk}}{k!}\zetax^k p^{\frac{k}{p^{n-1}(p-1)}}+\sum_{k=0}^{p-1}\frac{(-1)^{n(k+1)}}{k!}\zetax^{k+1}p^{\frac{k+p}{p^{n-1}(p-1)}}\sigma_n \\
                     & \quad-\sum_{k=1}^{p-1}\frac{H_k}{k!}(-1)^{n(k+1)}\zetax^{k+1} p^{\frac{k+p}{p^{n-1}(p-1)}}                                                                       \\
                     & \quad+\frac{1}{2}\zetax^2 p^{\frac{2}{p^{n-2}(p-1)}}\sigma_n^2+\frac{(-1)^n}{2}\zetax^3 p^{\frac{2}{p^{n-2}(p-1)}-\frac{p-2}{p^n(p-1)}}                          \\
                     & \quad+O\left(p^{\frac{2}{p^{n-2}(p-1)}}\right).
    \end{align*}
\end{proposition}

Thus, we have the $\left(\frac{2}{p^n(p-1)}-\frac{2}{p^n}\right)$-truncated expansion of $\left(\zeta_{p^{n+1}}-1\right)^{-2p+2}$:
\begin{proposition}\label{lem:fenmu}
    For an integer $n\geq 2$, we have
    \begin{align*}
          & \left(\zeta_{p^{n+1}}-1\right)^{-2p+2}                                                                                                                           \\
        = & p^{-\frac{2}{p^n}}\left(1-(-1)^n\zetax p^{\frac{1}{p^n(p-1)}}-\bbone_3(p)\cdot 3^{\frac{2}{3^n\cdot 2}}\sigma_{n+1}+O\left(p^{\frac{2}{p^n(p-1)}}\right)\right).
    \end{align*}
\end{proposition}
\begin{proof}
    For $p\geq 5$, by using the truncated expansion of $\zeta_{p^{n+1}}$, we have
    \begin{align*}
          & \left(\zeta_{p^{n+1}}-1\right)^{-2p+2}                                                                                                       \\
        = & \left((-1)^{n+1}\zetax p^{\frac{1}{p^n(p-1)}}+\frac{1}{2}\zetax^2 p^{\frac{2}{p^n(p-1)}}+O\left(p^{\frac{3}{p^n(p-1)}}\right)\right)^{-2p+2} \\
        = & p^{-\frac{2}{p^n}}\left(1-\frac{(-1)^n}{2}\zetax p^{\frac{1}{p^n(p-1)}}+O\left(p^{\frac{2}{p^n(p-1)}}\right)\right)^{-2p+2} .
    \end{align*}

    By using binomial expansion for negative exponent and truncation, we obtain
    \begin{align*}
          & \left(\zeta_{p^{n+1}}-1\right)^{-2p+2}                                                                                                                                        \\
        = & p^{-\frac{2}{p^n}}\sum_{k=0}^\infty \binom{-2p+2}{k}\left(-\frac{(-1)^n}{2}\zetax p^{\frac{1}{p^n(p-1)}}+O\left(p^{\frac{2}{p^n(p-1)}}\right)\right)^k                        \\
        = & p^{-\frac{2}{p^n}}\left(1+(-2p+2)\left(-\frac{(-1)^n}{2}\zetax p^{\frac{1}{p^n(p-1)}}+O\left(p^{\frac{2}{p^n(p-1)}}\right)\right)+O\left(p^{\frac{2}{p^n(p-1)}}\right)\right) \\
        = & p^{-\frac{2}{p^n}}\left(1-(-1)^n\zetax p^{\frac{1}{p^n(p-1)}}+O\left(p^{\frac{2}{p^n(p-1)}}\right)\right).
    \end{align*}

    For $p=3$, we have an extra term in the truncated expansion:
    \begin{align*}
        \zeta_{3^{n+1}}-1 = (-1)^{n+1}\zeta_4 3^{\frac{1}{3^n\cdot 2}}\left(1-\frac{(-1)^n}{2}\zeta_4 3^{\frac{1}{3^n\cdot 2}}+3^{\frac{2}{3^n\cdot 2}}\sigma_{n+1}+O\left(3^{\frac{2}{3^n\cdot 2}}\right)\right).
    \end{align*}
    Therefore, by power series expansion and truncation,
    \begin{align*}
        \left(\zeta_{3^{n+1}}-1\right)^{-4}= & 3^{-\frac{2}{3^n}}\left(1-\frac{(-1)^n}{2}\zeta_4 3^{\frac{1}{3^n\cdot 2}}+3^{\frac{2}{3^n\cdot 2}}\sigma_{n+1}+O\left(3^{\frac{2}{3^n\cdot 2}}\right)\right)^{-4} \\
        =                                    & 3^{-\frac{2}{3^n}}\left(1-(-1)^n\zeta_4 3^{\frac{1}{3^n\cdot 2}}-3^{\frac{2}{3^n\cdot 2}}\sigma_{n+1}+O\left(3^{\frac{2}{3^n\cdot 2}}\right)\right).
    \end{align*}
\end{proof}

\subsection{Truncated expansion of $A_{p,n}^{(\beta)}$}
In this paragraph, we give the $\left(\frac{2}{p^{n-1}(p-1)}\right)$-truncated expansion of $(A_{p,n}^{(\beta)})^k$, for $1\leq k\leq 2p-1$.

\begin{lemma}\label{lem:29041new}
    Let $n\geq 2$ be an integer. For $1\leq k\leq 2p-1$, we have
    \begin{align*}
        \left(A_{p,n}^{(\beta)}\right)^k= & (-1)^{nk}\zetax^k p^{\frac{k}{p^n(p-1)}}                                                                                                    \\
                                          & \quad+\bbone_{\leq p+1}(k)\cdot k(-1)^{nk}\zetax^k p^{\frac{k+p-1}{p^n(p-1)}}\sigma_{n+1}                                                   \\
                                          & \quad+\bbone_2(k)\cdot \zetax^2 p^{\frac{2}{p^{n-1}(p-1)}}\sigma_{n+1}^2+\bbone_3(k)\cdot 3(-1)^n\zetax^3 p^{\frac{2p^2-p+2}{p^{n+1}(p-1)}} \\
                                          & \quad+\bbone_{\leq p-1}(k)\cdot \beta k(-1)^{n(k+1)}\zetax^{k+1}p^{\frac{k+p}{p^n(p-1)}}                                                    \\
                                          & \quad+\bbone_{1}(k)\cdot \beta\zetax^2 p^{\frac{2}{p^{n-1}(p-1)}}\sigma_{n+1}+O\left(p^{\frac{2}{p^{n-1}(p-1)}}\right).
    \end{align*}
\end{lemma}
\begin{proof}
    By expanding the product and truncation, we obtain
    \begin{align*}
        \left(A_{p,n}^{(\beta)}\right)^k= & (-1)^{nk}\zetax^k p^{\frac{k}{p^{n-1}(p-1)}}\sigma_n^k+\beta k(-1)^{n(k+1)}\zetax^{k+1}p^{\frac{k+1}{p^{n-1}(p-1)}}\sigma_n^k \\
                                          & \quad+O\left(p^{\frac{2}{p^{n-1}(p-1)}}\right).
    \end{align*}
    Since the condition $\frac{k+1}{p^{n-1}(p-1)}-\frac{k}{p^n}<\frac{2}{p^{n-1}(p-1)}$ implies
    $k<p$, we can rewrite the expansion of $\left(A_{p,n}^{(\beta)}\right)^k$ as
    \begin{align*}
         & (-1)^{nk}\zetax^k p^{\frac{k}{p^{n-1}(p-1)}}\sigma_n^k+\bbone_{\leq p-1}(k)\cdot \beta k(-1)^{n(k+1)}\zetax^{k+1}p^{\frac{k+1}{p^{n-1}(p-1)}}\sigma_n^k+O\left(p^{\frac{2}{p^{n-1}(p-1)}}\right).
    \end{align*}
    The lemma follows from the following estimation of the terms:
    \begin{itemize}
        \item For the term $(-1)^{nk}\zetax^k p^{\frac{k}{p^{n-1}(p-1)}}\sigma_n^k$, by binomial expansion one has
              \begin{align*}
                  (-1)^{nk}\zetax^k p^{\frac{k}{p^{n-1}(p-1)}}\sigma_n^k=  (-1)^{nk}\zetax^k \sum_{t=0}^k\binom{k}{t}p^{\frac{k}{p^{n-1}(p-1)}-\frac{k-t}{p^n}}\sigma_{n+1}^t.
              \end{align*}
              The condition $\frac{k}{p^{n-1}(p-1)}-\frac{k-t}{p^n}-\frac{t}{p^{n+1}}<\frac{2}{p^{n-1}(p-1)}$ implies
              \[ t=
                  \begin{cases}
                      0,1,  & \text{ if }k\leq p+1,k\neq 2,3; \\
                      0,1,2 & \text{ if }k=2,3;               \\
                      0,    & \text{ if }p+2\leq k\leq 2p-1.
                  \end{cases}\]
              For $k=3$ and $t=2$, by truncation we have
              \begin{align*}
                   & (-1)^n\zetax^3 \binom{3}{2}p^{\frac{3}{p^n(p-1)}+\frac{2}{p^n}}\sigma_{n+1}^2= 3(-1)^n\zetax^3 p^{\frac{2p^2-p+2}{p^{n+1}(p-1)}}+O\left(p^{\frac{2}{p^{n-1}(p-1)}}\right).
              \end{align*}

              By combining these cases and simplify the exponent, we have
              \begin{align*}
                    & (-1)^{nk}\zetax^k p^{\frac{k}{p^{n-1}(p-1)}}\sigma_n^k                                                             \\
                  = & (-1)^{nk}\zetax^k p^{\frac{k}{p^n(p-1)}}+\bbone_{\leq p+1}(k)\cdot k(-1)^{nk}\zetax^k p^{\frac{k+p-1}{p^n(p-1)}}   \\
                    & \quad+\bbone_2(k)\cdot \zetax^2 p^{\frac{2}{p^{n-1}(p-1)}}\sigma_{n+1}^2                                           \\
                    & \quad+\bbone_3(k)\cdot 3(-1)^n\zetax^3 p^{\frac{2p^2-p+2}{p^{n+1}(p-1)}}+O\left(p^{\frac{2}{p^{n-1}(p-1)}}\right).
              \end{align*}
        \item For the term $\beta k(-1)^{n(k+1)}\zetax^{k+1}p^{\frac{k+1}{p^{n-1}(p-1)}}\sigma_n^k$ with $1\leq k\leq p-1$, by binomial expansion one has
              \begin{align*}
                    & \beta k(-1)^{n(k+1)}\zetax^{k+1}p^{\frac{k+1}{p^{n-1}(p-1)}}\sigma_n^k                                                  \\
                  = & \beta k(-1)^{n(k+1)}\zetax^{k+1}p^{\frac{k+1}{p^{n-1}(p-1)}}\sum_{t=0}^k\binom{k}{t}p^{-\frac{k-t}{p^n}}\sigma_{n+1}^t.
              \end{align*}
              The condition $\frac{k+1}{p^{n-1}(p-1)}-\frac{k-t}{p^n}-\frac{t}{p^{n+1}}<\frac{2}{p^{n-1}(p-1)}$ implies
              \[t=\begin{cases}
                      0,1, & \text{ if }k=1;             \\
                      0,   & \text{ if }2\leq k\leq p-1.
                  \end{cases}\]
              Therefore,
              \begin{align*}
                    & \beta k(-1)^{n(k+1)}\zetax^{k+1}p^{\frac{k+1}{p^{n-1}(p-1)}}\sigma_n^k                                                  \\
                  = & \bbone_{\leq p-1}(k)\cdot \beta k(-1)^{n(k+1)}\zetax^{k+1}p^{\frac{k+p}{p^n(p-1)}}                                      \\
                    & \quad+\bbone_{1}(k)\cdot \beta\zetax^2 p^{\frac{2}{p^{n-1}(p-1)}}\sigma_{n+1}+O\left(p^{\frac{2}{p^{n-1}(p-1)}}\right).
              \end{align*}
    \end{itemize}
\end{proof}

\begin{corollary}\label{coro:5955}
    For an integer $n\geq 2$, we have
    \begin{align*}
         & \zeta_{p^{n+1}}-\sum_{k=0}^{p-1}\frac{(-1)^k}{k!}\left(A_{p,n}^{(\beta)}\right)^k                                                                                               \\  = & (-1)^{n+1}\zetax p^{\frac{2p-1}{p^n(p-1)}}\sigma_{n+1}                                                                                                                     +\sum_{k=1}^{p-1}\frac{(-1)^k\left(k\beta-H_k\right)}{k!}(-1)^{nk+n+1}\zetax^{k+1} p^{\frac{k+p}{p^n(p-1)}}                                        \\
         & \quad+\beta\zetax^2 p^{\frac{2}{p^{n-1}(p-1)}}\sigma_{n+1}-\bbone_3(p)\cdot\frac{(-1)^n}{2}\zetax^3 p^{\frac{2p^2-p+2}{p^{n+1}(p-1)}}+O\left(p^{\frac{2}{p^{n-1}(p-1)}}\right).
    \end{align*}
\end{corollary}
\begin{proof}
    By \Cref{lem:29041new}, the summation $\sum_{k=0}^{p-1}\frac{(-1)^k}{k!}\left(A_{p,n}^{(\beta)}\right)^k$ can be expanded as:
    \begin{equation}\label{Rue}
        \begin{split}
            & \sum_{k=0}^{p-1}\frac{(-1)^k}{k!}\left(A_{p,n}^{(\beta)}\right)^k                                                                                                \\
            = & \sum_{k=0}^{p-1}\frac{(-1)^{(n+1)k}}{k!}\zetax^k p^{\frac{k}{p^n(p-1)}}+\sum_{k=1}^{p-1}\frac{k(-1)^{(n+1)k}}{k!}\zetax^k p^{\frac{k+p-1}{p^n(p-1)}}\sigma_{n+1} \\
            & \quad+\beta\sum_{k=1}^{p-1}\frac{k(-1)^{k+n(k+1)}}{k!}\zetax^{k+1}p^{\frac{k+p}{p^n(p-1)}}                                                                       \\
            & \quad+\frac{1}{2}\zetax^2 p^{\frac{2}{p^{n-1}(p-1)}}\sigma_{n+1}^2-\frac{1-\bbone_3(p)}{6}\cdot 3(-1)^n\zetax^3 p^{\frac{2p^2-p+2}{p^{n+1}(p-1)}}                \\
            & \quad-\beta\zetax^2 p^{\frac{2}{p^{n-1}(p-1)}}\sigma_{n+1}+O\left(p^{\frac{2}{p^{n-1}(p-1)}}\right),
        \end{split}
    \end{equation}
    where $1-\bbone_3(p)$ is the indicator function of the set of primes $p\geq 5$.

    Using the identity $\frac{k(-1)^{k+n(k+1)}}{k!}=-\frac{k(-1)^{(k+1)(n+1)}}{k!}$, we have
    $$\beta\sum_{k=1}^{p-1}\frac{k(-1)^{k+n(k+1)}}{k!}\zetax^{k+1}p^{\frac{k+p}{p^n(p-1)}}=-\sum_{k=1}^{p-1}\frac{k\beta(-1)^{(k+1)(n+1)}}{k!}\zetax^{k+1}p^{\frac{k+p}{p^n(p-1)}}.$$ Together with the shifting of the index in the summation
    $$\sum_{k=1}^{p-1}\frac{k(-1)^{(n+1)k}}{k!}\zetax^k p^{\frac{k+p-1}{p^n(p-1)}}\sigma_{n+1},$$ we can rewrite the summation (\ref{Rue}) as following
    \begin{align*}
          & \sum_{k=0}^{p-1}\frac{(-1)^k}{k!}\left(A_{p,n}^{(\beta)}\right)^k                                                                                                     \\
        = & \sum_{k=0}^{p-1}\frac{(-1)^{(n+1)k}}{k!}\zetax^k p^{\frac{k}{p^n(p-1)}}+\sum_{k=0}^{p-2}\frac{(-1)^{(n+1)(k+1)}}{k!}\zetax^{k+1} p^{\frac{k+p}{p^n(p-1)}}\sigma_{n+1} \\
          & \quad-\sum_{k=1}^{p-1}\frac{k\beta(-1)^{(k+1)(n+1)}}{k!}\zetax^{k+1}p^{\frac{k+p}{p^n(p-1)}}                                                                          \\
          & \quad+\frac{1}{2}\zetax^2 p^{\frac{2}{p^{n-1}(p-1)}}\sigma_{n+1}^2-\frac{(-1)^n(1-\bbone_3(p))}{2}\zetax^3 p^{\frac{2p^2-p+2}{p^{n+1}(p-1)}}                          \\
          & \quad-\beta\zetax^2 p^{\frac{2}{p^{n-1}(p-1)}}\sigma_{n+1}+O\left(p^{\frac{2}{p^{n-1}(p-1)}}\right).
    \end{align*}
    Therefore, by using the truncated expansion of $\zeta_{p^{n+1}}$ (cf. \Cref{truncatedfinal} l.c.) and combining terms, we have
    \begin{align*}
          & \zeta_{p^{n+1}}-\sum_{k=0}^{p-1}\frac{(-1)^k}{k!}\left(A_{p,n}^{(\beta)}\right)^k                                                                                          \\
        = & (-1)^{(n+1)}\zetax p^{\frac{2p-1}{p^n(p-1)}}\sigma_{n+1}+\sum_{k=1}^{p-1}\frac{(-1)^k\left(k\beta-H_k\right)}{k!}(-1)^{nk+n+1}\zetax^{k+1} p^{\frac{k+p}{p^n(p-1)}}        \\
          & \quad+\beta\zetax^2 p^{\frac{2}{p^{n-1}(p-1)}}\sigma_{n+1}-\bbone_3(p)\frac{(-1)^n}{2}\zetax^3 p^{\frac{2p^2-p+2}{p^{n+1}(p-1)}}+O\left(p^{\frac{2}{p^{n-1}(p-1)}}\right).
    \end{align*}
\end{proof}

\subsection{Proof of \Cref{thm:bestuniformizer}}
\begin{proposition}\label{prop:51912}
    For any integer $n\geq 2$, the $\frac{2}{p^n(p-1)}$-truncated expansion of
    \[\left(\zeta_{p^{n+1}}-1\right)^{-2p+2}\left(\zeta_{p^{n+1}}-\sum_{k=0}^{p-1}\frac{(-1)^k}{k!}\left(A_{p,n}^{(\beta)}\right)^k-\sum_{k=1}^{p-1}\frac{(-1)^k\left(k\beta-H_k\right)}{k!}\left(A_{p,n}^{(\beta)}\right)^{p+k}\right)\]
    is given by
    \begin{align*}
        (-1)^{(n+1)}\zetax p^{\frac{1}{p^n(p-1)}}\sigma_{n+1}+2\zetax^2 p^{\frac{2}{p^n(p-1)}}\sigma_{n+1}+O\left(p^{\frac{2}{p^n(p-1)}}\right).
    \end{align*}
\end{proposition}
\begin{proof}
    For $k=1,\cdots,p-1$, by \Cref{lem:29041new} we have
    \begin{align*}
        \left(A_{p,n}^{(\beta)}\right)^{p+k}=  (-1)^{nk+n+1}\zetax^{k+1}p^{\frac{k+p}{p^n(p-1)}}-\bbone_1(k)\cdot \zetax^2 p^{\frac{2}{p^{n-1}(p-1)}}\sigma_{n+1}.
    \end{align*}
    Therefore, together with \Cref{coro:5955}, one calculates
    \begin{align*}
          & \zeta_{p^{n+1}}-\sum_{k=0}^{p-1}\frac{(-1)^k}{k!}\left(A_{p,n}^{(\beta)}\right)^k-\sum_{k=1}^{p-1}\frac{(-1)^k\left(k\beta-H_k\right)}{k!}\left(A_{p,n}^{(\beta)}\right)^{p+k} \\
        = & (-1)^{(n+1)}\zetax p^{\frac{2p-1}{p^n(p-1)}}\sigma_{n+1}+\zetax^2 p^{\frac{2}{p^{n-1}(p-1)}}\sigma_{n+1}                                                                       \\
          & \quad-\bbone_3(p)\cdot\frac{(-1)^n}{2}\zetax^3 p^{\frac{2p^2-p+2}{p^{n+1}(p-1)}}+O\left(p^{\frac{2}{p^{n-1}(p-1)}}\right).
    \end{align*}
    Together with the identity (cf. \Cref{lem:fenmu} l.c.)
    \begin{align*}
          & \left(\zeta_{p^{n+1}}-1\right)^{-2p+2}                                                                                                                           \\
        = & p^{-\frac{2}{p^n}}\left(1-(-1)^n\zetax p^{\frac{1}{p^n(p-1)}}-\bbone_3(p)\cdot 3^{\frac{2}{3^n\cdot 2}}\sigma_{n+1}+O\left(p^{\frac{2}{p^n(p-1)}}\right)\right),
    \end{align*}
    we calculate by expanding the product and combining terms that
    \begin{align*}
          & \left(\zeta_{p^{n+1}}-1\right)^{-2p+2}                                                                                                                                                                \\
          & \quad\cdot\left(\zeta_{p^{n+1}}-\sum_{k=0}^{p-1}\frac{(-1)^k}{k!}\left(A_{p,n}^{(\beta)}\right)^k-\sum_{k=1}^{p-1}\frac{(-1)^k\left(k\beta-H_k\right)}{k!}\left(A_{p,n}^{(\beta)}\right)^{p+k}\right) \\
        = & (-1)^{(n+1)}\zetax p^{\frac{1}{p^n(p-1)}}\sigma_{n+1}+2\zetax^2 p^{\frac{2}{p^n(p-1)}}\sigma_{n+1}+O\left(p^{\frac{2}{p^n(p-1)}}\right)                                                               \\
          & \quad +\bbone_3(p)\cdot\left((-1)^n\zeta_4 3^{\frac{3/2}{3^n}}\cdot\sigma_{n+1}^2-(-1)^n\zeta_4 3^{\frac{5/6}{3^n}}\right.                                                                            \\
          & \phantom{\quad +\bbone_3(p)\cdot\left(\quad\quad\right.}\left.+3^{\frac{2}{3^n}}\cdot\sigma_{n+1}^2+(-1)^n\zeta_4 3^{\frac{11/6}{3^n}}\sigma_{n+1}-3^{\frac{4/3}{3^n}}\right).
    \end{align*}
    For $p\geq 5$, $\bbone_3(p)=0$ and we get the desired result in this case. For $p=3$, the result follows from the following calculation of extra terms:
    \begin{align*}
          & (-1)^n\zeta_4 3^{\frac{3/2}{3^n}}\cdot\sigma_{n+1}^2-(-1)^n\zeta_4 3^{\frac{5/6}{3^n}}+3^{\frac{2}{3^n}}\cdot\sigma_{n+1}^2+(-1)^n\zeta_4 3^{\frac{11/6}{3^n}}-3^{\frac{4/3}{3^n}} \\
        = & (-1)^n\zeta_4 3^{\frac{3/2}{3^n}}\cdot\left(3^{-\frac{2/3}{3^n}}+O\left(3^{-\frac{4/9}{3^n}}\right)\right)-(-1)^n\zeta_4 3^{\frac{5/6}{3^n}}                                       \\
          & \quad +3^{\frac{2}{3^n}}\cdot O\left(3^{-\frac{2/3}{3^n}}\right)+O\left(3^{\frac{1}{3^n}}\right)                                                                                   \\
        = & O\left(3^{\frac{1}{3^n}}\right).
    \end{align*}
\end{proof}

Recall that one has
\[{\pi_p^{2,1}}=p^{-1/p}\left(\zeta_{p^2}-\sum_{k=0}^{p-1}\frac{1}{[k!]}\zetax^k p^{\frac{k}{p(p-1)}}\right).\]
We deduce
the $\left(\frac{2}{p(p-1)}\right)$-truncated expansions of powers of ${\pi_p^{2,1}}$ from \Cref{truncatedfinal}:
\[\pi_{p}^{2,1}=\zeta_{2(p-1)}p^{
    \frac{1}{p(p-1)}}\sigma_2+\zetax^2 p^{\frac{2}{p(p-1)}}\sigma_2+O\left(p^{\frac{2}{p(p-1)}}\right).\]
Take $A_{p,2}^{(1)}=\pi_p^{2,1}$, by \Cref{prop:51912} we know that
\begin{align*}
    \pi_p^{3,1}=-\zetax p^{\frac{1}{p^2(p-1)}}\sigma_3+2\zetax^2 p^{\frac{2}{p^2(p-1)}}\sigma_3+O\left(p^{\frac{2}{p^2(p-1)}}\right).
\end{align*}
Assume we have proved that
\[\pi_p^{n,1}=(-1)^n\zetax p^{\frac{1}{p^{n-1}(p-1)}}\sigma_n+2\zetax^2 p^{\frac{2}{p^{n-1}(p-1)}}\sigma_n+O\left(p^{\frac{2}{p^{n-1}(p-1)}}\right),\]
then by taking $A_{p,n}^{(2)}=\pi_p^{n,1}$, one can deduce from \Cref{prop:51912} that
\[\pi_p^{n+1,1}=(-1)^{(n+1)}\zetax p^{\frac{1}{p^n(p-1)}}\sigma_{n+1}+2\zetax^2 p^{\frac{2}{p^n(p-1)}}\sigma_{n+1}+O\left(p^{\frac{2}{p^n(p-1)}}\right).\]
Therefore, for every $m\geq 3$, we obtain
\[v_p\left(\pi_p^{m,1}\right)=\frac{1}{p^{m-1}(p-1)}-\frac{1}{p^m}=\frac{1}{p^m(p-1)}=e_{\bbK_p^{m,1}/\bbQ_p}^{-1}.\]
Thus, we can conclude that
$\pi_p^{m,1}$ is a uniformizer of $\bbK_p^{m,1}$ for every $m\geq 3$.

\section{MacLane's pseudo-valuations}
In this section, unless otherwise specified, $K$ will be a finite extension of ${\bbQ_p}$ with uniformizer $\pi_K$, residue field $\kappa_K$, value group $\Gamma_K$ and the normalized $p$-adic valuation $v_p$ (i.e. $v_p(p)=1$). We recall that MacLane's pseudo-valuations are defined as following:
\begin{definition}
    A map $v$ from the $K$-algebra $K[T]$ to $\bbR\cup\{\infty\}$ is called a \textbf{pseudo-valuation} if the following conditions are satisfied: for all $f,g\in K[T]$,

    \noindent\begin{enumerate*}
        \item $v(T)\geq 0$;
        \item $v(f+g)\geq \min(v(f),v(g))$;
        \item $v(f\cdot g)=v(f)+v(g)$;
        \item $v\vert_K=v_p$.
    \end{enumerate*}
\end{definition}
Let $V(K[T])$ be the set of all pseudo-valuations on $K[T]$. For any $v\in V(K[T])$, we define the value group $\Gamma_v$ of $v$ to be the Grothendieck group of the monoid $v(K[T])\backslash\{\infty\}$.

\subsection{Pseudo-valuations and Berkovich unit disc}
Recall that the Berkovich unit disc $\DBerk$ on $K$ is the set of all bounded multiplicative semi-norms on the $K$-Tate algebra
$$K\langle T\rangle\coloneqq\left\{\sum_{k=0}^\infty a_k T^k\in K\llbracket T\rrbracket\middle\vert\lim_{k\to\infty} v_p(a_k)=+\infty\right\}.$$

\begin{definition}[{\cite[B.5]{baker_potential_2010}}]
    Let $(T,\leq)$ be a partially ordered set. We call $(T,\leq)$ a \textbf{parametrized rooted tree} if there exists a function $\alpha:T\to\bbR_{\geq 0}$, satisfying the following axioms:
    \begin{enumerate}[label={(P\arabic*)}]
        \item $T$ has a unique maximal element $\zeta$, called the root of $T$.
        \item For each $x\in T$, the set $\{z\in T\vert z\geq x\}$ is totally ordered.
        \item $\alpha(\zeta)=0$.
        \item $\alpha$ is order-reversing, in the sense that $x\leq y$ implies $\alpha(x)\geq \alpha(y)$.
        \item The restriction of $\alpha$ to any full totally ordered subset of $T$ gives a bijection onto a real interval. (A totally ordered subset $S$ of $T$ is called full if $x,y\in S,z\in T$, and $x\leq z\leq y$ implies $z\in S$.)
    \end{enumerate}
\end{definition}
\begin{remark}
    There is a more geometric notion which is equivalence to the parametrized rooted tree, called $\bbR$-tree (cf. \cite[Section 1.4]{baker_potential_2010}). For our purpose, the description of parametrized rooted tree is more convenient.
\end{remark}

There is a partial order $``\leq" $ on $\DBerk$ defined as following:
Let $\norm{\cdot}_x,\norm{\cdot}_y\in\DBerk$. We say $\norm{\cdot}_x\leq \norm{\cdot}_y$, if $\norm{f}_x\leq\norm{f}_y$ for every $f\in K\langle T\rangle$.
\begin{theorem}[{cf. \cite[Section 1.4]{baker_potential_2010}}]
    The partially ordered set $(\DBerk,\leq)$ is a parametrized rooted tree with its root: the Gauss norm
    $$\norm{\cdot}_\frakG:\sum_{k=0}^\infty a_k T^k\mapsto \max_{k\geq 0} \left(p^{-v_p(a_k)}\right).$$
\end{theorem}

For $\norm{\cdot}_x\in\DBerk$, let $\underline{x}$ be the norm induced by $\norm{\cdot}_x$ on the field $\calH(x)\coloneqq\mathrm{Frac}(K[T]/\ker \norm{\cdot}_x)$. Denote by $\kappa_{\underline{x}}\coloneqq \underline{x}^{-1}\left(\left[0,1\right]\right)/\underline{x}^{-1}\left(\left[0,1\right)\right)$ the residue field of $\underline{x}$ and by $\Gamma_{\underline{x}}$ the value group of $\underline{x}$. Define $\frake_x\coloneqq\dim_\bbQ(\Gamma_{\underline{x}}/\Gamma_K\otimes \bbQ)$ and $\frakf_x\coloneqq\mathrm{tr}\deg( \kappa_{\underline{x}}/\kappa_K)$. Points in $\DBerk$ can be classified using the parameters $\frake_x$, $\frakf_x$ and $\calH(x)$ (cf. \cite[Definition 2.3.3.3]{ducros_introduction_2015}):
\begin{enumerate}    \item The point $x$ is of Type I, if $\calH(x)\subseteq\bbC_p$;
    \item The point $x$ is of Type II, if we have $\frake_x=0$ and $\frakf_x=1$;
    \item The point $x$ is of Type III, if we have $\frake_x=1$ and $\frakf_x=0$;
    \item The point $x$ is of Type IV, if $\frake_x=\frakf_x=0$ and $x$ is not of Type I.
\end{enumerate}

\begin{proposition}\label{thm:corres}
    The points in $V(K[T])$ are in one-to-one correspondence to the points in the Berkovich unit disc $\DBerk$.
\end{proposition}
\begin{proof}
    The restriction of an element $\norm{\cdot}_x\in\DBerk$ to $K[T]$ induces a pseudo-valuation on $K[T]$:
    $$v_x: K[T]\to\bbR\cup\{\infty\},\ f\mapsto -\log_p(\norm{f}_x).$$
    Conversely, let $v_y\in V(K[T])$ be a pseudo-valuation. It gives a multiplicative semi-norm on $K[T]$:
    $$\norm{\cdot}_y: K[T]\to\bbR\cup\{\infty\},\ f\mapsto p^{-v_y(f)}.$$
    Given $f=\sum_{k=0}^\infty a_k T^k\in K\langle T\rangle$, we set
    $$\norm{f}_y\coloneqq \lim_{n\to\infty}\norm*{\sum_{k=0}^n a_k T^k}_y.$$
    This limit converges by the strong triangle inequality. One can verify that this gives a bounded multiplicative semi-norm on $K\langle T\rangle$, i.e. an element in $\DBerk$.

    It is immediate to check that the maps
    $$V(K[T])\to \DBerk,\ v_y\mapsto \norm{\cdot}_y$$
    and
    $$\DBerk\to V(K[T]),\ \norm{\cdot}_x\mapsto v_x$$
    are inverse to each other.
\end{proof}

Using this bijection, $V(K[T])$ is endowed with a partial order induced by that on $\DBerk$. Moreover, the type classification of points on $\DBerk$ induces a type classification of points on $V(K[T])$.

\begin{example}[{pseudo-valuations on $K[T]$}]
    \begin{enumerate}
        \item The \textbf{Gauss valuation} $$v_\frakG:K[T]\to \bbR\cup\{0\}, \sum_{i=0}^n a_k t^k\mapsto \min_{0\leq i\leq n}v_p(a_k)$$
              is a pseudo-valuation of Type II.
        \item For an irreducible polynomial $G(T)\in K[T]$, the $p$-adic valuation on $K$ extends to the $p$-adic valuation $v_L=v_p$ on $L\coloneqq K[T]/(G)$. Then
              $$v:K[T]\to\bbR\cup\{\infty\},\ f\mapsto v_L(\bar{f})$$
              is a pseudo-valuation of Type I with nontrivial kernel.
    \end{enumerate}
\end{example}

\subsection{Augmentation and inductive pseudo-valuations}

MacLane has a method to augment a pseudo-valuation in $(V(K[T]), \leq )$, which is based on  a special class of irreducible polynomials in $K[T]$, called \textbf{ key polynomials}. To define the key polynomials, we need the following notions:
\begin{definition}[{cf. \cite[I.2, Definition 4.1]{MacLane1936a}}]\label{def:3781}
    Let $v\in V(K[T])$ and $f,g\in K[T]$.
    \begin{enumerate}
        \item Say $f$ and $g$ are \textbf{$v$-equivalent}, which is denoted by $f\sim_v g$, if $v(f-g)>v(f)$ or $f=g=0$.
        \item Say $f$ is \textbf{$v$-divisible} by $g$, which is denoted by $g\vert_v f$, if there exists $q\in K[T]$ such that $f\sim_v qg$.
        \item A polynomial $\phi\in K[T]$ is \textbf{$v$-irreducible} if for any $f,g\in K[T]$, we have $\phi \vert_v fg\Rightarrow \phi \vert_v f$ or $\phi\vert_v g$.
        \item A non-constant polynomial $\phi\in K[T]$ is \textbf{$v$-minimal} if for every $f\in K[T]\backslash\{0\}$, we have $\phi\vert_v f\Rightarrow \deg(\phi)\leq \deg(f)$.
    \end{enumerate}
\end{definition}
Now we can define the key polynomial over a pseudo-valuation $v$ as following:
\begin{definition}[{cf. \cite[Definition 4.1]{MacLane1936a}}]\label{def:17617}
    A $v$-irreducible and $v$-minimal monic polynomial $\phi\in K[T]$ is called a \textbf{key polynomial} over $v$.
\end{definition}

Given a pseudo-valuation $v\in V(K[T])$, a key polynomial $\phi\in K[T]$ over $v$ and a key value $\lambda\in \bbR\cup\{\infty\}$ with $\lambda>v(\phi)$, MacLane defines the \textbf{augmentation} of $v$ associated to $\phi$ and $\lambda$ as following:
$$w:K[T]\to\bbR\cup\{\infty\},\ f\mapsto \min_{0\leq i\leq m}v(a_i)+i\lambda,$$
where $f=\sum_{i=0}^m a_i\phi^i$ is the $\phi$-adic expansion of $f$. This is a pseudo-valuation which we denoted by $w=[v,w(\phi)=\lambda]$.

\begin{definition}[{cf. \cite[Definition 6.1]{MacLane1936a}}]
    For a pseudo-valuation $v\in V(K[T])$, if there exists a sequence of pseudo-valuations $v_0=v_\frakG,v_1,\cdots,v_k=v$ satisfying:
    \begin{enumerate}
        \item $v_i=[v_{i-1},v_i(\phi_i)=\lambda_i]$ for $i=1,\cdots k$;
        \item $\deg(\phi_{i+1})\geq \deg(\phi_i)$ for $i=1,\cdots,k-1$;
        \item $\phi_{i+1}\nsim_{v_i}\phi_i$ for $i=1,\cdots,k-1$,
    \end{enumerate}
    then we say that $v$ has a $k$-th \textbf{(inductive) representation}
    $$\rho_v=[v_\frakG,v_1(\phi_1)=\lambda_1,\cdots,v_k(\phi_k)=\lambda_k]$$
    with key polynomials $\phi_1,\cdots,\phi_k$ and key values $\lambda_1,\cdots,\lambda_k$.
\end{definition}

If a pseudo-valuation $v$ has an inductive representation $\rho_v$, then we say that it is an \textbf{inductive pseudo-valuation} (represented by $\rho_v$) or it is inductive in short.
The representation of an inductive pseudo-valuation is not necessarily unique. This can be observed by the following fact:
\begin{lemma}[{cf. \cite[Lemma 15.1]{MacLane1936a}}]\label{lem:52781}
    If $$[v_\frakG,v_1(\phi_1)=\lambda_1,\cdots,v_k(\phi_k)=\lambda_k]$$ is a representation of an inductive pseudo-valuation $v$ with $\deg\phi_{k-1}=\deg\phi_k$, then
    $$[v_\frakG,v_1(\phi_1)=\lambda_1,\cdots,v_{k-2}(\phi_{k-2})=\lambda_{k-2},v_k(\phi_k)=\lambda_k]$$
    is also a representation of $v$.
\end{lemma}

Among the representations of an inductive valuation, there is a special representation, called homogeneous representation, defined as following:
\begin{definition}[{cf. \cite[Section 16]{MacLane1936a}}]
    Let \[\rho_v=[v_\frakG,v_1(\phi_1)=\lambda_1,\cdots,v_k(\phi_k)=\lambda_k]\]be a $k$-th representation of an inductive valuation $v$ in $V(K[T])$. Call $\rho_v$ a \textbf{homogeneous (inductive) representation}, if for every $i=1,\cdots,k$, the key polynomial $\phi_i$ can be written as\footnote{We define $\phi_0=T$.}
    $$\phi_i=\sum_{j} c_j\cdot\phi_0^{m_0^{(j)}}\cdots\phi_{i-1}^{m_{i-1}^{(j)}},$$
    satisfying
    \begin{enumerate}
        \item $c_j\in \left\{[u]\cdot\pi_K^t\middle\vert u\in \kappa_K, t\geq 0\right\}$ for all $j$, where $[u]$ is the Teichm\"{u}ller lift of $u$;
        \item if $i>1$, then $m_l^{(j)}<\deg(\phi_{l+1})/\deg(\phi_l)$ for $j,l=0,\cdots,i-2$;
        \item $v_{i-1}\left(c_j\cdot\phi_0^{m_0^{(j)}}\cdots\phi_{i-1}^{m_{i-1}^{(j)}}\right)=v_{i-1}(\phi_{i})$ for all $j$.
    \end{enumerate}
\end{definition}
By \cite[Theorem 16.3, Theorem 16.4]{MacLane1936a}, we know that:
\begin{enumerate}
    \item Every inductive pseudo-valuation $v\in V(K[T])$ has a unique homogeneous representation $\rho_v^{\mathsf{h}}$.
    \item $v=w$ if and only if $\rho_v^{\mathsf{h}}$ and $\rho_w^{\mathsf{h}}$ are identical, i.e. they share the same length, key polynomials and key values.
\end{enumerate}

The following result characterizes the type classification in terms of inductive valuations, and shows that Type IV points are not inductive.

\begin{proposition}[{cf. \cite[Corollary 1.116]{micu_pseudovaluations_2020}}]\label{prop:16750}
    Let $v\in V(K[T])$ be a pseudo-valuation.
    \begin{enumerate}
        \item $v$ is of Type I if and only if $v$ is not a valuation.
        \item The Gauss valuation is of Type II.
        \item Inductive valuations consist of Type I, Type II (except for $v_\frakG$) and Type III points. More precisely, let
              $$\rho_v^{\mathsf{h}}=[v_\frakG,v_1(\phi_1)=\lambda_1,\cdots,v_k(\phi_k)=\lambda_k]$$
              be the homogeneous representation of $v$. Then
              \begin{enumerate}
                  \item $v$ is of Type I if and only if $\lambda_k=\infty$;
                  \item $v$ is of Type II if and only if $\lambda_k\in\bbQ$;
                  \item $v$ is of Type III if and only if $\lambda_k\in\bbR\backslash\bbQ$.
              \end{enumerate}
    \end{enumerate}
\end{proposition}

\section{Recurrence polynomials via pseudo-valuations}

\begin{lemma}\label{lem:23861}
    For any element $A$ in $\calO_{\bbK_p^{m,n}}$ and rational number $r\geq 0$, there exists a polynomial $R$ over $\bbZ$ such that $v_p\left(A-R\left(\pi_p^{m,n}\right))\right)>r$.
\end{lemma}
\begin{proof}
    Since $\bbK_p^{m,n}$ is totally ramified over $\bbQ_p$, $A\in\calO_{\bbK_p^{m,n}}$ can be written as
    $$A=\sum_{k=0}^\infty a_k \left(\pi_p^{m,n}\right)^k,\ a_k\in\{0,1,\cdots,p-1\}.$$
    Set $R(T)=\sum_{k=0}^{\lceil r/e_{m,n}\rceil}a_k T^k$. Then
    $$v_p\left(A-R\left(\pi_p^{m,n}\right)\right)=v_p\left(\sum_{k=\lceil r/e_{m,n}\rceil+1}^\infty a_k\left(\pi_p^{m,n}\right)^k\right)>r.$$
\end{proof}

In rest of this paragraph, we prove \Cref{thm:40345} by induction using the pseudo-valuations. By \cite{WangYuan2021}, the uniformizer $\pi_p^{m-1,n}$ of $\bbK_p^{m-1,n}$ is constructed for $m=3$. We suppose that the uniformizer $\pi_p^{m-1,n}$ of $\bbK_p^{m-1,n}$ is constructed for some $m\geq 3$.

Fix an integer $n\geq 1$. Let $G_m(T)\coloneqq T^p-\zeta_{p^{m-1}}$ and $e_{m,n}\coloneqq \frac{1}{p^{m+n-1}(p-1)}$ be the $p$-adic valuation of any uniformizer of $\bbK_p^{m,n}$. We have an isomorphism of $p$-adic fields:
$$\bbK_p^{m-1,n}[T]/(G_m)\to\bbK_p^{m,n}, T\mapsto \zeta_{p^m}.$$
Thus, the $p$-adic valuation on $\bbK_p^{m,n}$ corresponds to a pseudo-valuation $v^{m-1}\in V\left(\bbK_p^{m-1,n}[T]\right)$ with kernel $(G_m)$. Note $v_p\left(\zeta_{p^m}-1\right)=p^n\cdot e_{m,n}$. Thus, by B\'{e}zout lemma, to construct an element of $\bbK_p^{m,n}$ with $p$-adic valuation $e_{m,n}$, it is enough to find a polynomial over $\bbK_p^{m-1,n}$ which is mapped to $d\cdot e_{m,n}$  by $v^{m-1}$, with some $d>0$ and $\gcd(d,p)=1$.

By \Cref{prop:16750}, $v^{m-1}$ has an inductive representation since it is of Type I. The following proposition collects some properties of the homogeneous representation of $v^{m-1}$:
\begin{proposition}
    Let
    $$\rho_v^{\mathsf{h}}=[v_\frakG,v_1(\phi_1)=\lambda_1,\cdots,v_k(G_m)=\infty]$$
    be the homogeneous representation of $v^{m-1}$.
    \begin{enumerate}[ref=(\arabic*)]
        \item \label{it:14854}$\phi_1=T-1$;
        \item \label{it:60497}for every $i=1,\cdots k$, $\deg\phi_i\in\{1,p\}$.
    \end{enumerate}
    Moreover, let $s\geq 1$ be the index of the last key polynomial in $\rho_v^{\mathsf{h}}$ with degree $1$, then we have
    \begin{enumerate}[resume,ref=(\arabic*)]
        \item \label{it:48926}$\rho_v^\prime=[v_\frakG,w_1(\phi_s)=\lambda_s,w_2(G_m)=\infty]$ is also an inductive representation of $v^{m-1}$;
        \item \label{it:2375} $\lambda_s=d\cdot e_{m,n}$ for a positive integer $d$ satisfying $\gcd\left(d,p\right)=1$.
    \end{enumerate}
\end{proposition}
\begin{proof}
    \begin{enumerate}
        \item One calculates that
              \begin{align*}
                    & v_\frakG\left(G_m(T)-(T-1)^p\right)
                  =  v_\frakG\left(\sum_{t=1}^{p-1}(-1)^{t+1}\binom{p}{t}T^t+(1-\zeta_{p^{m-1}})\right)                                                 \\
                  = & \min\left(v_p\left((-1)^p\binom{p}{p-1}\right),\cdots,v_p\left((-1)^2\binom{p}{1}\right),v_p\left(1-\zeta_{p^{m-1}}\right)\right) \\
                  > & 0=v_\frakG(G_m(T)),
              \end{align*}
              which means that $G_m(T)$ is $v_\frakG$-equivalent to $(T-1)^p$. Since $T-1$ is a homogeneous key polynomial over $v_\frakG$, by \cite[Section 9]{MacLane1936a} one knows that $\phi_1=T-1$.
        \item  For $i=1,\cdots,k-1$,
              let $$G_m(T)=g_{m_i}^{(i)}\cdot\phi_i^{m_i}+\cdots+g_{1}^{(i)}\phi_i+g_0^{(i)}$$
              be the $\phi_i$-adic expansion of $G_m$ and let
              $$\calN_i\coloneqq \left\{j\in \{0,\cdots,m_i\}: v_i\left(g_j^{(i)}\phi_i^j\right)=v_i(G_m)\right\}.$$
              By \cite[Theorem 5.2]{MacLane1936b} and \cite[Theorem 5.3]{MacLane1936b}, one has
              $$\left(\max\calN_i-\min\calN_i\right)\cdot \deg\phi_i =\deg G_m=p.$$
              Then $\deg\phi_k\mid p$.
        \item Such $s$ exists by the assertions \ref{it:14854} and \ref{it:60497}. We have $\deg\phi_1=\cdots=\deg\phi_s=1$ and $\deg\phi_{s+1}=\cdots=\deg\phi_k$. The result follows from applying \Cref{lem:52781} repeatedly.
        \item Since $v^{m-1}$ corresponds to the $p$-adic valuation on $\bbK_p^{m,n}$, we know that $\Gamma_v=\Gamma_{\bbK_p^{m,n}}=e_{m,n}\bbZ$. By the assertion $\ref{it:48926}$, $\Gamma_v$ is generated by $\Gamma_{v_\frakG}=\Gamma_{\bbK_p^{m-1,n}}=e_{m-1,n}\bbZ$ and $\lambda_s$. The result follows.
    \end{enumerate}
\end{proof}

Let $s\geq 1$ be the index of the last key polynomial in $\rho_v^{\mathsf{h}}$ with degree $1$ as in the previous proposition.
Since $\phi_s$ is a homogeneous key polynomial over $v_{s-1}$ of degree $1$, it is monic  with coefficients in $\calO_{\bbK_p^{m-1,n}}$, i.e. $\phi_s=T-A_m$, $A_m\in \calO_{\bbK_p^{m-1,n}}$. Thus, by the assertion \ref{it:2375} of the previous proposition, there exists a positive integer $d$ with $\gcd(d,p)=1$ such that
$$v_p\left(\zeta_{p^m}-A_m\right)=\lambda_s=d\cdot e_{m,n}.$$
By \Cref{lem:23861}, there exists a polynomial $\calR_p^{m,n}(T)\in\bbZ_{(p)}[T]$ such that $v_p\left(A_m-\calR_p^{m,n}(\pi_p^{m-1,n})\right)>d\cdot e_{m,n}$, thus
$v_p\left(\zeta_{p^m}-\calR_p^{m,n}(\pi_p^{m-1,n})\right)=d\cdot e_{m,n}$.
By B\'{e}zout lemma, there exist two integers  $\alpha_p^{m,n},\beta_p^{m,n}$ that
\begin{align*}
    e_{m,n}= & \alpha_p^{m,n}\cdot p^n\cdot e_{m,n}+\beta_p^{m,n}\cdot d\cdot e_{m,n}                                                                    \\
    =        & \alpha_p^{m,n}\cdot v_p\left(\zeta_{p^m}-1\right)+\beta_p^{m,n}\cdot v_p\left(\zeta_{p^m}-\calR_p^{m,n}\left(\pi_p^{m-1,n}\right)\right).
\end{align*}
Thus,
$$\pi_p^{m,n}\coloneqq \left(\zeta_{p^m}-1\right)^{\alpha_p^{m,n}}\cdot\left(\zeta_{p^m}-\calR_p^{m,n}\left(\pi_p^{m-1,n}\right)\right)^{\beta_p^{m,n}}$$
is a uniformizer of $\bbK_p^{m,n}$.

\begin{remark}
    We actually proved that $\calR_p^{m,n}(T)$ can be chosen in $\bbZ[T]$. Compared to the statement of \Cref{thm:40345} that $\calR_p^{m,n}(T)\in\bbZ_{(p)}[T]$, the result is essentially equivalent in the sense of \Cref{lem:23861}. We keep $\bbZ_{(p)}[T]$ in the statement to be consistent with \Cref{thm:bestuniformizer}.
\end{remark}

\end{document}